\documentclass{amsart}
\usepackage{amssymb}
\usepackage{mathtools}
\usepackage{enumerate}
\numberwithin{equation}{section}
\DeclarePairedDelimiter\abs{\lvert}{\rvert}%
\DeclarePairedDelimiter\norm{\lVert}{\rVert}%
\DeclarePairedDelimiter\bnorm{\biggl\lVert}{\biggr\rVert}%
\makeatletter
\let\oldabs\abs
\renewcommand{\abs}{\@ifstar{\oldabs}{\oldabs*}}
\let\oldnorm\norm
\renewcommand{\norm}{\@ifstar{\oldnorm}{\oldnorm*}}
\makeatother

\theoremstyle{plain}
\newtheorem{thm}{Theorem}[section]
\newtheorem{theorem}{Theorem}

\newtheorem{lem}[thm]{Lemma}
\newtheorem{prop}[thm]{Proposition}

\theoremstyle{definition}

\theoremstyle{remark}

\newcommand\R{\mathbb{R}}

\newcommand\N{\mathbb{N}}
\newcommand\C{\mathbb{C}}

\newcommand\T{\mathcal{T}}
\newcommand\K{\mathcal{K}}
\newcommand\M{\mathcal{M}}

\newcommand{\ben}{\begin{enumerate}[(i)]}
\newcommand{\een}{\end{enumerate}}
\newcommand{\ft}{\mathcal{F}}

\newcommand{\ftd}{\mathcal{F}_{\mathbb{R}^d}}

\newcommand{\ift}{\mathcal{F}^{-1}}

\newcommand{\h}{\mathcal{H}_d}

\newcommand{\les}{\lesssim}
\newcommand{\mt}{\tilde{\mu}_d}
\newcommand{\md}{\mu_d}
\newcommand{\mds}{d\mu_d(s)}

\begin{document}
\title[Maximal Operators Associated with Radial Fourier Multipliers]{A Characterization of Maximal Operators Associated with Radial Fourier Multipliers}
\subjclass[2000]{42B15, 42B25}
\author{Jongchon Kim}
\address{Department of Mathematics, University of Wisconsin-Madison, Madison, WI,
53706 USA}
\email{jkim@math.wisc.edu}
\begin{abstract} 
We give a simple necessary and sufficient condition for maximal operators associated with radial Fourier multipliers to be bounded on $L^p_{rad}$ and $L^p$ for certain $p$ greater than $2$. The range of exponents obtained for the $L^p_{rad}$ characterization is optimal for the given condition. The $L^p$ characterization is derived from an inequality of Heo, Nazarov, and Seeger regarding a characterization of radial Fourier multipliers.
\end{abstract}
\maketitle

\section{Introduction}
Let $T_m$ be a Fourier multiplier transformation on $\R^d$ defined by 
\begin{equation*}
\widehat{T_mf} = m \hat{f}
\end{equation*}
for a function $m$. There are sufficient conditions for $L^p$ boundedness of $T_m$ such as Mikhlin-H\"{o}rmander multiplier theorem, but a necessary and sufficient condition is known only for $p=1$ and $p=2$ (see e.g. \cite{StSin}). However, Garrig\'{o}s and Seeger \cite{GaSe} obtained a striking characterization of radial multiplier transformations that are bounded on the subspace $L^p_{rad}$ of radial $L^p$ functions for $1<p<\frac{2d}{d+1}$, which is optimal for the characterization. By duality, this also implies a characterization result in the dual range $\frac{2d}{d-1}<p'<\infty$. More recently, Heo, Nazarov, and Seeger \cite{HNS} proved that the necessary and sufficient condition given in \cite{GaSe} actually gives a characterization of radial multiplier transformations that are bounded on the entire $L^p$ space, provided that the dimension is sufficiently large. 

Furthermore, for the maximal operators $M_m$ defined by
\begin{equation*}
M_mf(x) = \sup_{t>0} |T_{m(\cdot/t)} f(x)|,
\end{equation*}
Garrig\'{o}s and Seeger \cite{GaSeM} gave a necessary and sufficient condition for the $L^p_{rad}$ boundedness of $M_m$ for radial $m$ compactly supported away from the origin. Assume that $m$ is supported in $\{\xi: \frac{1}{2} < |\xi| < 2\}$. Let $\ft$ be the Fourier transform, $K_t = \ift [m(\cdot/t)]$ be the associated convolution kernel, $1<p<\frac{2d}{d+1}$, $p\leq q \leq \infty$, and $I=[1,2]$. It was shown in \cite{GaSeM} that $M_m$ extends to a bounded operator from $L^p_{rad}$ to $L^{p,q}_{rad}$ if and only if
\begin{equation} \label{eqn:conGSM}
\bnorm{\sup_{t\in I} |K_t| }_{L^{p,q}} < \infty,
\end{equation}
where $L^{p,q}=L^{p,q}(\R^d)$ is the standard Lorentz space.

A radial multiplier of great interest is the (truncated) Bochner-Riesz multiplier 
\begin{equation*}
m^\lambda(\xi) = (1-|\xi|^2)_+^\lambda \chi(|\xi|)
\end{equation*} 
for a suitable cut-off function $\chi$ supported in $[1/2,2]$. We note that the $L_{rad}^p$ boundedness of the maximal Bochner-Riesz operators for the optimal $p$-range was obtained by Kanjin \cite{Kan}. See also \cite{Col}.

It is natural to ask if there is a characterization of maximal operators $M_m$ that are bounded on $L^{p'}_{rad}$ for the dual range $\frac{2d}{d-1}< p'<\infty$. We answer the question in the affirmative. Here and in the following statements, we let $K = \ift [m]$, and the multiplier $m$ is always assumed to be bounded, radial, and supported in $\{\xi\in \R^d: \frac{1}{2} < |\xi| < 2\}$.
\begin{thm} \label{thm:radial}
Let $d\geq 2$, $1<p<\frac{2d}{d+1}$, and $p \leq q \leq \infty$. Then $M_m$ extends to a bounded operator from $L^{p',q'}_{rad}$ to $L_{rad}^{p'}$ if and only if 
\begin{equation} \label{eqn:nece}
\norm{K}_{L^{p,q}} < \infty. 
\end{equation}
\end{thm}

In fact, our result shows that
\begin{equation*}
\norm{T_m}_{L^p_{rad} \to L^{p,q}_{rad}}  \approx \norm{M_m}_{L^{p',q'}_{rad} \to L^{p'}_{rad}} \approx \norm{K}_{L^{p,q}}
\end{equation*}
for the $p,q$ range of Theorem \ref{thm:radial} \footnote{We note that $\norm{T_m}_{L^p_{rad} \to L^{p,q}_{rad}}  \approx \norm{K}_{L^{p,q}}$ is known in a greater generality (see \cite{GaSe}).}, where we denote by $A \approx B$ the inequality $C^{-1} B \leq A \leq C B$ for an absolute constant $C=C_{d,p,q}>0$.

Theorem \ref{thm:radial} was stated originally with the condition
\begin{equation} \label{eqn:nece2} \sup_{\norm{g}_{L^1(I)}=1} \norm{\int_I g (t) K_t dt }_{L^{p,q}} < \infty.
\end{equation}
That we may replace \eqref{eqn:nece2} by \eqref{eqn:nece} was kindly pointed out to us by Andreas Seeger.

Since \eqref{eqn:nece} is weaker than \eqref{eqn:conGSM}, we see that $L_{rad}^{p}$ bounds of $M_m$ for $1<p<\frac{2d}{d+1}$ implies $L_{rad}^{p'}$ bounds as well.  This ``dual" result also relies on the kernel estimate given in \cite{GaSe}.

Let us return to the example $m^\lambda$, whose kernel $K^\lambda = \ft^{-1}[m^\lambda]$ satisfies
\begin{equation*}
|K^\lambda(x)| \leq C (1+|x|)^{-(\frac{d+1}{2} + \lambda)}.
\end{equation*}
Then, for $p=\frac{2d}{d+1+2\lambda}$, $M_{m^\lambda}$ extends to a bounded operator from $L_{rad}^p$ to $L_{rad}^{p,\infty}$ by \cite{GaSeM} and from $L_{rad}^{p',1}$ to $L_{rad}^{p'}$ by Theorem \ref{thm:radial}. Real interpolation between the two endpoint estimates recovers the result of Kanjin \cite{Kan}.

It was pointed out in \cite{GaSeM} that the characterization by the condition \eqref{eqn:conGSM} is false if $L^p_{rad}$ is replaced by $L^p$. A counter example was given by Tao \cite{TaoB}, regarding the maximal Bochner-Riesz operators for $p<2$.

On the other hand, we observe that an important inequality of \cite{HNS}, which implies a characterization of radial Fourier multipliers, also implies a characterization of maximal operators $M_m$ that are bounded on the entire $L^{p'}$ space. 

\begin{thm}\label{thm:Max}
Let $1<p<\frac{2(d-1)}{d+1}$, $p\leq q\leq \infty$, and $d\geq 4$. Then $M_m$ extends to a bounded operator from $L^{p',q'}$ to $L^{p'}$ if and only if $\norm{K}_{L^{p,q}} < \infty$.
\end{thm}

We emphasize that this result does not improve on the range of exponents in the endpoint $L^{p',1} \to L^{p'}$ bounds for the maximal Bochner-Riesz operators established in \cite{LRS} for the Stein-Tomas range $1<p<\frac{2(d+1)}{d+3}$. For the pointwise convergence of Bochner-Riesz means, see \cite{CRV} and also \cite{LS} for an endpoint result.

Before we state the inequality from \cite{HNS}, we recall some notations. Let $\sigma_r$ be the surface measure on the $(d-1)$-dimensional sphere of radius $r$ centered at the origin. Let $\psi_0$ be a smooth radial function compactly supported in $\{x:|x|\leq 1/10\}$ whose Fourier transform vanishes at the origin to a high order but not on $\{ \xi: 1/8 \leq  |\xi| \leq 8 \}$. Set $\psi = \psi_0*\psi_0$.

\begin{theorem}[\cite{HNS}] \label{thm:HNS} Let $1<p< \frac{2(d-1)}{d+1}$ and $d\geq 4$. Then
\begin{equation} \label{eqn:HNS}
\norm{ \int_{\R^d} \int_1^\infty h(y,r) \sigma_r*\psi(\cdot - y) drdy }_{L^p}^p \leq C 
\int_{\R^d} \int_1^\infty |h(y,r)|^p r^{d-1} dr dy.
\end{equation}
\end{theorem}

A consequence of Theorem \ref{thm:HNS} is the estimate $\norm{T_m}_{L^p \to L^{p,q}} \approx \norm{K}_{L^{p,q}}$ (see Section 5 and 9 of \cite{HNS}). Without much additional work, we strengthen this implication as follows. 

\begin{lem} \label{prop:MaxHNS} Assume that \eqref{eqn:HNS} holds for $1<p< p_d$ for some $p_d$. Let $1<p<p_d$ and $p \leq q \leq \infty$. Then
\begin{equation*}
\norm{M_m}_{L^{p',q'} \to L^{p'}} \approx \norm{K}_{L^{p,q}}.
\end{equation*}
\end{lem}
We remark that there are other important applications of \eqref{eqn:HNS} including local smoothing estimates for the wave equation (see \cite{HNS,HNS2}). 

This paper is organized as follows. Lemma \ref{prop:MaxHNS} is proved in Section \ref{sec:MaxHNS}. In Section \ref{sec:hankel}, Theorem \ref{thm:radial} is formulated in a slightly general setting in terms of maximal Hankel multiplier transformations. Proof of Theorem \ref{thm:radial} shall be given in Section \ref{sec:kernel} and \ref{sec:proof}. In what follows, we shall frequently use the standard notation $\les$ instead of $\leq C$ for an absolute constant $C$.

\section{Proof of Lemma \ref{prop:MaxHNS}} \label{sec:MaxHNS}
We start with some standard reduction. By Littlewood-Paley theory, it is sufficient to consider the maximal operator $\tilde{M}_m$ where the sup is taken over $t\in I:=[1,2]$ using the fact that the exponent $p'$ is greater than $2$ (see Section 3 of \cite{LRS}). By duality, this amounts to prove 
\begin{equation}\label{eqn:golHNS}
\norm{\int_I K_t * g_t dt}_{L^{p,q}} \les \norm{K}_{L^{p,q}} \norm{\int_I |g_t| dt}_{L^p}.
\end{equation}

We argue as in Section 5 and 9 of \cite{HNS}. Since $m$ is radial, $K_t:=\ift [m(\cdot/t)]$ is also radial. Let $\kappa_t$ be the function such that $K_t(\cdot) = \kappa_t(|\cdot|)$. Fix a Schwartz function $\eta_0$ such that $\hat \eta_0$ is compactly supported in $\{\xi: \frac{1}{8} < |\xi| < 8 \}$ and $\hat \eta_0 (\xi) = 1$ for $\frac{1}{4}< |\xi|<4$. The function $\eta_0$ was chosen so that $\hat \eta_0 (\xi)  m(\xi/t) = m(\xi/t)$ for $t\in I$. Let $\eta$ be a Schwartz function defined by $\hat \eta = \hat{\eta_0} (\hat{\psi})^{-1}$. By our choice of $\eta_0$, we may write $K_t* g_t = \psi * K_t * f_t$, where $f_t = \eta*g_t$. Observe that
\begin{equation} \label{eqn:obsrt}
\begin{split}
\int_I \psi * K_t * f_t  dt &= \int_I \left[\int_{\R^d} \int_0^\infty \psi*\sigma_r(\cdot-y)  \kappa_t(r) f_t(y) dr dy \right]dt \\
&= \int_{\R^d} \int_0^\infty \psi*\sigma_r(\cdot-y) h(y,r) dr dy,
\end{split}
\end{equation}
where we take $h(y,r) = \int_I  \kappa_t(r) f_t(y) dt$.

Thus, we are already in the situation to apply \eqref{eqn:HNS}. However, we would like to have
\begin{equation} \label{eqn:HNSm}
\norm{ \int_{\R^d} \int_0^\infty \sigma_r*\psi(\cdot - y)h(y,r)  drdy }_{L^p} \les 
\norm{ h}_{L^p(\R^d\times [0,\infty) ; dyr^{d-1} dr)},
\end{equation}
where the $r$-integration is taken over $[0,\infty)$. This issue was circumvented in \cite{HNS} by splitting the kernel $K$. For the convenience of having \eqref{eqn:HNSm}, we shall show that \eqref{eqn:HNSm} follows from \eqref{eqn:HNS} by crude estimations. For this, it is enough to show that
\begin{equation} \label{eqn:lower}
\norm{ \int_{\R^d} \int_0^1 \sigma_r*\psi(\cdot - y) h(y,r)  drdy }_{L^p} \les
\norm{h}_{L^p( \R^d\times [0,1]; dyr^{d-1}dr)}.
\end{equation}
Let $h_r(y) = h(y,r)$. Then the left hand side of \eqref{eqn:lower} is bounded by
\begin{align*}
\norm{ \int_0^1 \sigma_r*\psi* h_r  dr}_{L^p} &\les  
\int_0^1  \norm{ \sigma_r*\psi* h_r }_{L^p} dr  \\
&\les \int_0^1  \norm{ \sigma_r*\psi}_{L^1} \norm{h_r }_{L^p} dr
\les \int_0^1 r^{d-1}  \norm{h_r }_{L^p} dr.
\end{align*} 
Finally, H\"{o}lder's inequality gives \eqref{eqn:lower}.

By real interpolation and a Lorentz space estimate, \eqref{eqn:HNSm} implies
\begin{equation}\label{eqn:HNSv}
\norm{ \int_{\R^d} \int_0^\infty \sigma_r*\psi(\cdot - y)h(y,r)  drdy }_{L^{p,q}} 
\les
\left(\int_{\R^d} \norm{h(y,\cdot)}_{L^{p,q}([0,\infty);r^{d-1}dr)}^p dy\right)^{1/p}
\end{equation} 
for $1<p<p_d$ and $p\leq q \leq \infty$. We refer the reader to Section 9 of \cite{HNS}.

Now we turn to \eqref{eqn:obsrt}. We have
\begin{equation} \label{eqn:crucial}
\norm{h(y,\cdot)}_{L^{p,q}([0,\infty);r^{d-1}dr)} \les \norm{K}_{L^{p,q}} \int_I  |f_t(y)| dt
\end{equation}
by Minkowski's inequality for the Banach space $L^{p,q}$. Thus,
\begin{equation*}
\norm{\int_I K_t * g_t dt}_{L^{p,q}} \les \norm{K}_{L^{p,q}} \norm{\int_I |f_t| dt}_{L^p} \les \norm{K}_{L^{p,q}}  \norm{\int_I |g_t| dt}_{L^p}
\end{equation*}
by \eqref{eqn:obsrt} and \eqref{eqn:HNSv} which finishes the proof of \eqref{eqn:golHNS}.

\section{Hankel multipliers} \label{sec:hankel}
We would like to formulate Theorem \ref{thm:radial} in a slightly more general context. For a real number $d\geq 1$, let $\h$ be the modified Hankel transform (Fourier-Bessel transform) acting on functions $f$ on $\R_+:=(0,\infty)$ defined by
\begin{equation*}
\h f(\rho) = \int_0^\infty B_d(\rho r) f(r) d\mu_d(r),
\end{equation*}
where $B_d(x) = x^{-\frac{d-2}{2}} J_{\frac{d-2}{2}} (x)$, $J_\alpha$ denotes the standard Bessel function of order $\alpha$, and $\mu_d$ is the measure on $\R_+$ given by $d\mu_d(r)=r^{d-1}dr$. 
If $F$ is a radial function on $\R^d$ given by $F(\cdot) = f(|\cdot|)$ for some function $f$, then $\ftd[F](\xi) = (2\pi)^d\h f(|\xi|)$ (see \cite{StWe}).

Let $m$ be a bounded function supported in $[\frac{1}{2},2]$. Consider the maximal operator $\M_m$ defined by 
\begin{equation*}
\M_m f(r) = \sup_{t>0} |\T_{m(\cdot/t)} f (r)|,
\end{equation*}
where $\T_m$ is the Hankel multiplier transformation defined by $\h[\T_mf] = m \h f$. When $d\in \N$, one may identify the Hankel multiplier transformation $\T_m$ with the Fourier multiplier transformation $T_{m(|\cdot|)}$ acting on radial functions. More precisely, we have $T_{m(|\cdot|)} F(x) = \T_m f (|x|)$ if $F(\cdot) = f(|\cdot|)$.

\begin{thm} \label{thm:radialhan}
Let $d>1$, $1<p<\frac{2d}{d+1}$, and $p \leq q \leq \infty$. Then we have
\begin{equation*}
\norm{\M_m}_{L^{p',q'}(\mu_d) \to L^{p'}(\mu_d)} \approx 
\norm{\h[m]}_{L^{p,q}(\mu_d)}.
\end{equation*}
\end{thm}
Theorem \ref{thm:radial} is a special case of Theorem \ref{thm:radialhan}, since for $d\in \N$
\begin{equation*}
K(x) = \ft_{\R^d}^{-1}[m(|\cdot|)](x) = c_d \h m(|x|).
\end{equation*}

Let $\mt$ be the measure on $\R$ given by $\mt (x) = (1+|x|)^{d-1}dx$. The estimate
\begin{equation}\label{eqn:fre}
\norm{\M_m}_{L^{p',q'}(\mu_d) \to L^{p'}(\mu_d)} \les A(p,q) := \norm{(1+|\cdot|)^{-\frac{d-1}{2}} \ft_{\R}^{-1}[m] }_{L^{p,q}(\mt)}
\end{equation}
implies Theorem \ref{thm:radialhan}, since
\begin{equation*}
A(p,q) \les \norm{\h[m]}_{L^{p,q}(\mu_d)} \les \norm{\T_m}_{L^{p}(\mu_d) \to L^{p,q}(\mu_d)} \les \norm{\M_m}_{L^{p',q'}(\mu_d) \to L^{p'}(\mu_d)}.
\end{equation*}
See \cite{GaSe} for the first two inequalities. The last inequality holds by duality.

Let $I=[1,2]$. By an argument similar to the one in Section \ref{sec:MaxHNS}, \eqref{eqn:fre} is equivalent to the dual statement 
\begin{equation}\label{eqn:golrad}
\norm{\int_I \T_{m(\cdot/t)} f_t dt}_{L^{p,q}(\mu_d)} \les A(p,q) \norm{\int_I |f_t| dt}_{L^p(\mu_d)}.
\end{equation}
The rest of the paper is devoted to the proof of \eqref{eqn:golrad}.

\section{Kernel estimate} \label{sec:kernel}
Let $K(r,s)$ be a bounded linear operator from $L^1(I)$ to $\C$ given by
\begin{equation*}
K(r,s)[g] = \int_0^\infty \int_I m(\rho/t) g(t) dt B_d(r\rho) B_d(s\rho) d\md(\rho)
\end{equation*}
for $g\in L^1(I)$. Then one may write
\begin{equation*}
\int_I \T_{m(\cdot/t)} [f_t](r) dt = \int_0^\infty K(r,s)[f(s)] \mds.
\end{equation*}
Here, we regard $f$ as a function on $\R_+$ taking values in $B:=L^1(I)$. 

Let $\kappa_t = \ft_{\R}^{-1}[m(\cdot/t)]$ and $\K [f(s)](x) = \int_I  f_t(s) \kappa_t(x)dt$. Then we have
\begin{equation*}
K(r,s)[f(s)] = \int_0^\infty \ft_{\R}[\K[f(s)]](\rho) \eta(\rho) B_d(r\rho) B_d(s\rho) \mds,
\end{equation*}
where $\eta$ is a smooth function supported in $[\frac{1}{8},8]$ such that $\eta(\cdot)m(\cdot/t) = m(\cdot/t)$ for all $t\in I$. Next, we use the kernel estimate, Proposition 3.1 in \cite{GaSe}, which implies that 
\begin{equation} \label{eqn:kernel}
|K(r,s)[f(s)]| \les_N \sum_{\pm,\pm} \frac{W[f(s)](\pm r \pm s)}{[(1+r)(1+s)]^\frac{d-1}{2}},
\end{equation}
where $W[f(s)](x) = |\K[f(s)]| * w_N(x)$ for $w_N(x) = (1+|x|)^{-N}$ for $N>1$.

We have the following analogue of \eqref{eqn:crucial}.
\begin{equation} \label{eqn:W}
\begin{split}
\norm{\frac{W[f(s)]}{(1+|\cdot|)^{\frac{d-1}{2}}} }_{L^{p,q}(\mt)} &\les \norm{\frac{\K[f(s)]}{(1+|\cdot|)^{\frac{d-1}{2}}} }_{L^{p,q}(\mt)} \les A(p,q) |f(s)|_B.
\end{split}
\end{equation}
See Lemma 2.3 in \cite{GaSe} for the first inequality. 

Furthermore, by a Lorentz space version of H\"{o}lder's inequality (see \cite{Oneil})
\begin{equation} \label{eqn:WL1}
\begin{split}
\norm{W[f(s)]}_{L^1(\R)} &\les \norm{\K[f(s)]}_{L^1(\R)} \\
&= \int (1+|x|)^{-\frac{d-1}{2}} (1+|x|)^{-\frac{d-1}{2}} |\K[f(s)](x)| d\mt(x) \\
&\les \norm{(1+|\cdot|)^{-\frac{d-1}{2}}}_{L^{p',1}(\mt) }  \norm{\frac{\K[f(s)]}{(1+|\cdot|)^{\frac{d-1}{2}}} }_{L^{p,\infty}(\mt)} \\
&\les A(p,\infty) |f(s)|_B.
\end{split}
\end{equation}
We remark that $A(p,\infty) \leq A(p,q)$. For the last inequality, $p'>\frac{2d}{d-1}$ is used.

In addition, we have 
\begin{align*}
\norm{\frac{W[f(s)]}{(1+|\cdot|)^{\frac{d-1}{2}}}}_{L^\infty} \les \norm{\K[f(s)]}_{L^\infty} &= \norm{\K[f(s)]*\ift[\eta]}_{L^\infty} \\ &\les \norm{\K[f(s)]}_{L^1(\R)} \les A(p,\infty) |f(s)|_B
\end{align*}
by \eqref{eqn:WL1}. Interpolating this estimate and \eqref{eqn:W} (with $q=\infty$) gives that for $\sigma>p$,
\begin{equation} \label{eqn:sigma}
\norm{\frac{W[f(s)]}{(1+|\cdot|)^{\frac{d-1}{2}}} }_{L^{\sigma}(\mt)} \les A(p,\infty) |f(s)|_B.
\end{equation}

\section{Proof of \eqref{eqn:golrad}} \label{sec:proof}
\subsection{Decomposition}
Let $I_m=[2^m,2^{m+1})$ and $I_m^*=[2^{m-2},2^{m+3})$.
As in \cite{GaSe}, we decompose 
\begin{equation}
\begin{split}
&\int_0^\infty K(r,s)[f(s)] \mds \\
&=\sum_m [\chi_{(0,2^{m-2})}(r) +  \chi_{I_m^*}(r)+ \chi_{[2^{m+3},\infty)}(r)] \int_{I_m} K(r,s)[f(s)] \mds.
\end{split}
\end{equation}

Then by the kernel estimate \eqref{eqn:kernel},
$$\abs{\int_0^\infty K(r,s)[f(s)] \mds} \les Ef(r) + \sum_m S_m f(r)+ Hf(r),$$
where 
\begin{align*}
E f (r) &= \sum_{\pm,\pm}
\int_{4r}^\infty \frac{W[f(s)](\pm r\pm s)}{[(1+r)(1+s)]^\frac{d-1}{2}} \mds,\\
S_m f(r)&=\chi_{I_m^*}(r) \sum_{\pm,\pm}\int_{I_m} \frac{W[f(s)](\pm r\pm s)}{[(1+r)(1+s)]^\frac{d-1}{2}} \mds,\\
H f(r) &= \sum_{\pm,\pm} \int_{0}^{r/4} \frac{W[f(s)](\pm r\pm s)}{[(1+r)(1+s)]^\frac{d-1}{2}} \mds.
\end{align*}

In the following sections, we shall omit the summation notation $\sum_{\pm,\pm}$. We shall prove that the operators $E$, $\sum_m S_m$, and $H$ are bounded from $L^{p,q}(\mu_d,B)$ to $L^{p,q}(\mu_d)$, thus actually proving an estimate shaper than \eqref{eqn:golrad}.

\subsection{Estimation of $Hf$}
\begin{prop} Let $d>1$, $1<p<\frac{2d}{d+1}$, and $1\leq q \leq \infty$. Then 
\begin{equation*}
\norm{Hf}_{L^{p,q}(\mu_d)} \les A(p,q) \norm{f}_{L^{p,\infty}(\mu_d,B)}.
\end{equation*}
\end{prop}
\begin{proof}
By Minkowski's inequality,
\begin{equation} \label{eqn:hf}
\norm{Hf}_{L^{p,q}(\mu_d)} \leq \int_0^\infty \norm{ \frac{\chi_{(4s,\infty)} W[f(s)](\pm \cdot \pm s)}{(1+\cdot)^\frac{d-1}{2}}}_{L^{p,q}(\mu_d)} \frac{\mds}{(1+s)^\frac{d-1}{2}}.
\end{equation}
By the change of variable $r \to r \pm s$ and the fact that $r\geq 4s$, the norm inside of the integral is bounded by
\begin{equation*}
\norm{\frac{\chi_{(3s,\infty)} W[f(s)](\pm \cdot)}{(1+\cdot)^\frac{d-1}{2}}}_{L^{p,q}(\mu_d)} 
\les \norm{\frac{W[f(s)]}{(1+|\cdot|)^\frac{d-1}{2}}}_{L^{p,q}(\tilde{\mu_d})} 
\les A(p,q) |f(s)|_B,
\end{equation*}
by \eqref{eqn:W}. Thus \eqref{eqn:hf} is bounded by
\begin{align*}
A(p,q) \int_0^\infty |f(s)|_B \frac{\mds}{(1+s)^\frac{d-1}{2}} &\leq A(p,q) \norm{f}_{L^{p,\infty}(\mu_d,B)} \norm{(1+\cdot)^{-\frac{d-1}{2}}}_{L^{p',1}(\mu_d)} \\
&\les A(p,q) \norm{f}_{L^{p,\infty}(\mu_d,B)}.
\end{align*}
For the last inequality, the condition $p' > \frac{2d}{d-1}$ was used.
\end{proof}

\subsection{Estimation of $S_mf$}
\begin{prop} \label{prop:sm} Let $d>1$, $1<p<\frac{2d}{d+1}$, and $1\leq q\leq \infty$. Then 
\begin{equation*}
\bnorm{\sum_m S_mf }_{L^{p,q}(\mu_d)} \les A(p,\infty) \norm{f}_{L^{p,q}(\mu_d,B)}.
\end{equation*}
\end{prop}
Proposition \ref{prop:sm} is a consequence of the following lemma by real interpolation.
\begin{lem} \label{lem:sm} Let $d$ and $p$ as in Proposition \ref{prop:sm}. Then 
\begin{equation*}
\bnorm{\sum_m S_mf }_{L^{\sigma}(\mu_d)} \les A(p,\infty) \norm{f}_{L^\sigma(\mu_d,B)}
\end{equation*}
for $\sigma=1$ and $p<\sigma<\frac{2d}{d+1}$.
\end{lem}
\begin{proof}
Considering the support of $S_m f$, it is enough to show that
\begin{equation}\label{eqn:goalsm}
\norm{S_mf }_{L^{\sigma}(\mu_d)} \les A(p,\infty) \norm{\chi_{I_m}f}_{L^\sigma(\mu_d,B)}.
\end{equation}

Since $r\sim s$, $S_m f(r)$ is bounded by
\begin{equation*}
\chi_{I_m^*}(r)\int W[\chi_{I_m}f(s)](\pm r\pm s) ds
= \chi_{I_m^*}(r)\int W[\chi_{I_m}f(s\pm r )](\pm s) ds.
\end{equation*}

By Minkowski's inequality,
\begin{equation}\label{eqn:sm1}
\norm{S_m f}_{L^\sigma(\mu_d)} \les 2^{m(d-1)/\sigma} \int \norm{W[\chi_{I_m} f(\cdot)](s)}_{L^\sigma(\R)} ds,
\end{equation}
after a change of variable in the $r$-variable. 

Although we have used Minkowski's inequality for the change of variable, we should integrate in the $s$-variable first in order to apply the estimates from Section \ref{sec:kernel}. For the case $\sigma=1$, we apply Fubini's theorem and \eqref{eqn:WL1} to bound \eqref{eqn:sm1} by a constant times
\begin{equation*}
A(p,\infty) \int_{I_m} 2^{m(d-1)} |f(r)|_B dr \les A(p,\infty) \norm{\chi_{I_m} f}_{L^1(\mu_d,B)},
\end{equation*}
which gives \eqref{eqn:goalsm}.

For the case $p<\sigma<\frac{2d}{d+1}$, we apply H\"{o}lder's inequality to interchange the order of the integration. Using $\norm{(1+|\cdot|)^{-(d-1)(\frac{1}{\sigma}-\frac{1}{2})}}_{L^{\sigma'}(\R)} < \infty$ and \eqref{eqn:sigma}, \eqref{eqn:sm1} is bounded by
\begin{align*}
C\left(2^{m(d-1)} \int \norm{ \frac{W[\chi_{I_m} f(r)]}{(1+|\cdot|)^{(d-1)/2} }}_{L^\sigma(\tilde{\mu}_d)}^\sigma dr\right)^{1/\sigma} \les A(p,\infty) \norm{\chi_{I_m}f}_{L^\sigma(\mu_d,B)}.
\end{align*}
\end{proof}

\subsection{Estimation of $Ef$}
\begin{prop} \label{prop:em} Let $d>1$, $1<p<\frac{2d}{d+1}$, and $1\leq q\leq \infty$. Then 
\begin{equation*}
\norm{E f }_{L^{p,q}(\mu_d)} \les A(p,\infty) \norm{f}_{L^{p,q}(\mu_d,B)}.
\end{equation*}
\end{prop}
Proposition \ref{prop:em} is a consequence of the following lemma by real interpolation.
\begin{lem} Let $d$ and $p$ as in Proposition \ref{prop:em}. Then 
\begin{equation*}
\norm{Ef }_{L^{\sigma}(\mu_d)} \les A(p,\infty) \norm{f}_{L^\sigma(\mu_d,B)}
\end{equation*}
for $\sigma=1$ and $p<\sigma<\frac{2d}{d+1}$.
\end{lem}
\begin{proof}
By the change of variable $s\to s\pm r$,
\begin{equation*}
Ef(r) \les \int_{3r}^\infty W[f(s\pm r)](\pm s) r^{-(d-1)/2} s^{(d-1)/2} ds.
\end{equation*}
By Minkowski's inequality
\begin{align*}
\norm{Ef}_{L^\sigma(\mu_d)} &\les \int_0^\infty \left( \int_0^{s/3} [W[f(s\pm r)](\pm s)]^\sigma r^{(d-1)(1-\frac{\sigma}{2})}dr \right)^{1/\sigma} s^{(d-1)/2} ds \\
&\les \int_0^\infty \left( s^{(d-1)(1-\frac{\sigma}{2})} \int_0^{s/3} [W[f(s\pm r)](\pm s)]^\sigma dr \right)^{1/\sigma} s^{(d-1)/2} ds \\
&\les \int_0^\infty \left( s^{d-1}  \int_{2s/3}^{4s/3} [W[f(r)](\pm s)]^\sigma dr \right)^{1/\sigma} ds \\
&\les \int_0^\infty \left( \int_{0}^{\infty} [W[f(r)](\pm s)]^\sigma r^{d-1}dr \right)^{1/\sigma} ds.
\end{align*}

The rest of the proof is similar to that of Lemma \ref{lem:sm}. Indeed, we apply Fubini's theorem and \eqref{eqn:WL1} for $\sigma=1$, and H\"{o}lder's inequality and \eqref{eqn:sigma} for $p<\sigma<\frac{2d}{d+1}$.
\end{proof}

\subsection*{Acknowledgement}
This paper will be a part of the author's PhD thesis. Part of this research was carried out while the author was visiting the Hausdorff Research Institute for Mathematics (HIM) in Bonn. The author would like to thank HIM and the organizers of the trimester program on Harmonic Analysis and Partial Differential Equations for their hospitality and generous support during the visit. This work was supported in part by the National Science Foundation.

\providecommand{\bysame}{\leavevmode\hbox to3em{\hrulefill}\thinspace}
\providecommand{\MR}{\relax\ifhmode\unskip\space\fi MR }
\providecommand{\MRhref}[2]{%
  \href{http://www.ams.org/mathscinet-getitem?mr=#1}{#2}
}
\providecommand{\href}[2]{#2}

\end{document}